\documentclass{article}

\usepackage{amsmath,amssymb,amsthm}

\usepackage[T1]{fontenc}

\renewcommand{\thispagestyle}[1]{} 
\newtheorem{theorem}{Theorem}[section]
\newtheorem{corollary}[theorem]{Corollary}
\newtheorem{lemma}[theorem]{Lemma}
\newtheorem{proposition}[theorem]{Proposition}

\newtheorem{definition}[theorem]{Definition} \theoremstyle{remark}
\newtheorem{remark}[theorem]{Remark} 

\numberwithin{equation}{section} 

\newcommand{\comment}[1]{}

\newcommand{\R}{\mathbb R}

\newcommand{\EE}{\mathbb{E}}
\newcommand{\PP}{\mathbb{P}}

\newcommand{\eps}{\varepsilon}
\newcommand{\la}{\lambda}

\newcommand{\ls}{\leqslant}
\newcommand{\gr}{\geqslant}

\newcommand{\conv}{{\rm conv}}
\renewcommand{\span}{{\rm span}}

\providecommand{\Probb}[2]{\mathbb{P}_{#1}\left(#2\right)}

\providecommand{\abs}[1]{\lvert#1\rvert}
\providecommand{\norm}[1]{\left\lVert#1\right\rVert}

\begin{document}
\large \title{Volume of the polar of random sets and shadow systems}

\author{D. Cordero-Erausquin \and M. Fradelizi \\
\and 
  G. Paouris \and P. Pivovarov}

\date{November, 2013}

\maketitle
\begin{abstract}
We obtain optimal inequalities for the volume of the polar of random
sets, generated for instance by the convex hull of independent random
vectors in Euclidean space. Extremizers are given by random vectors
uniformly distributed in Euclidean balls. This provides a random
extension of the Blaschke-Santal\'o inequality which, in turn, can be
derived by the law of large numbers.  The method involves generalized
shadow systems, their connection to Busemann type inequalities, and
how they interact with functional rearrangement inequalities.
\end{abstract}

\section{Introduction}

A celebrated result of Blaschke and Santal\'o~\cite{Santalo} states
that among symmetric convex bodies $K$ of fixed volume in the
Euclidean space $(\R^n, |\cdot|)$, the volume of the polar body
$K^\circ$ is maximized by the Euclidean ball, and therefore also by
ellipsoids, by $SL_n$-invariance (precise definitions will be recalled
below, in \S 2); we refer to~\cite{MP_BS} for a proof based on Steiner
symmetrization.  In the present paper, we are interested in extending
such a result to random sets.  A typical example of such a random set
is given by the convex hull of the columns of a random matrix, for
which we can prove the following property.

\begin{theorem}
  \label{thm:random-santalo} 
  Let $N,n \ge 1$.  In the class of $N$-tuples $(X_1,\ldots, X_N)$ of
  independent random vectors in $\R^n$ whose laws have a density
  bounded by one, the expectation of the volume of the set
$$\big(\conv\{\pm X_1,\ldots, \pm X_N\}\big)^{\circ}$$ is maximized by
  $N$ independent random vectors uniformly distributed in the
  Euclidean ball $D_n\subset \R^n$ of volume one.
\end{theorem}

The density of a measure on $\R^n$ will always refer to the density
with respect to Lebesgue measure on $\R^n$ (so it is implicit that the
measure is absolutely continuous with respect to Lebesgue
measure).

To see how the latter theorem generalizes the Blaschke-Santal\'o
inequality, let $K$ be a symmetric convex body and assume, without
loss of generality, that $\abs{K}=1$, where $\abs{\cdot}$ denotes
Lebesgue measure.  Let $X_1,\ldots, X_N, \ldots$ be a sequence of
independent random vectors uniformly distributed in $K$; by this we
mean that the law of $X_i$ is $\lambda_K$, Lebesgue measure restricted
to $K$, which has density $1_K$ (indeed bounded by one).  It is known
that $\conv\{\pm X_1,\ldots, \pm X_N\}$ converges almost surely to
$K$, in the Hausdorff metric, as $N\to +\infty$.  The latter also
holds in the special case when $K=D_n$. Consequently, we derive from
the theorem above that, in the limit, $\abs{K^{\circ}}\le
\abs{D_n^{\circ}}$ under the assumption that $\abs{K}=\abs{D_n}=1$,
which is the Blaschke-Santal\'o inequality; more detailed arguments as
well as other applications will be given in \S\ref{section:apps}.

Our work is part of the program initiated in~\cite{PaoPiv_probtake}
aimed at obtaining systematic random extensions (i.e. for random sets)
of several inequalities in convexity.  In the present paper, we treat
inequalities that are \emph{dual} to those considered
in~\cite{PaoPiv_probtake} (i.e.  inequalities for the polar bodies,
such as the Blaschke-Santal\'o inequality). Steiner symmetrization,
and more precisely \emph{shadow systems}, as in the work of Campi and
Gronchi~\cite{CG_product}, which was our main source of inspiration,
will play a central role, together with rearrangement inequalities.

In fact, we prove a general inequality, in the spirit of those
established in \cite{PaoPiv_probtake}. Our main result below extends
the statement of the previous theorem in several ways:
\begin{itemize}
\item[(i)] the result holds in distribution, not only in expectation;
\item [(iii)] we can replace Lebesgue measure by any rotationally
  invariant, radially decreasing, measure (for instance we can
  consider the volume of the intersection with a fixed Euclidean
  ball);
\item [(iii)] we can perform more general (convex) operations than the
  convex hull.
\end{itemize}

Before stating the result, we need to introduce a bit of
notation. Given $N$ vectors $x_1, \ldots, x_N$ in some $\R^n$ space,
we form the $n \times N$ matrix $[x_1\cdots x_N]$ that we view as an
operator from $\R^N$ to $\R^n$ or rather to
$\textrm{span}\{x_i\}\subset \R^n$; therefore if $C$ is a set in
$\R^N$, we denote
  \begin{equation*}
  [x_1\cdots x_N]C = \Bigl\{\sum_{i=1}^N c_i \, x_i: c=(c_i)_{i\le
    N}\in C \Bigr\} \subset \textrm{span}\{x_i\}\subset \R^n.
\end{equation*} 
Accordingly, if $ X_1,\ldots,X_N$ are random vectors in $\R^n$, the
random matrix $[X_1\cdots X_N]$ is a random linear operator from
$\R^N$ to $\R^n$ and for $C\subset \R^N$, $[X_1 \cdots X_N]C$ is a
random set in $\R^n$.
   
A convex body $C$ in $\R^N$ is \emph{unconditional} if it is invariant
under the coordinate reflections, i.e. $(c_1, \ldots, c_N)\in
C\Rightarrow (\pm c_1, \ldots, \pm c_N) \in C$. Typical examples
include the unit ball $B_p^N=\{c\in \R^N \; : \ \sum |c_i|^p \le 1\}$
of the $\ell_p^N$ space, for $p\ge 1$.
   
Finally, let us denote by $\mathcal{P}_n$ the class of all Borel
probability measures on $\R^n$ that have an $L^1$-density with respect
to Lebesgue measure bounded by $1$, i.e. with some abuse of notation,
$$\mathcal P_n = \Big\{ \mu
\; : \ d\mu(x)= f(x)\, dx \ \textrm{ with } f\ge 0, \int\! f = 1
\textrm{ and } \norm{f}_\infty \le 1 \Big\}.$$ where
$\norm{\cdot}_{\infty}$ is the essential supremum.  This set includes
Lebesgue measure restricted to sets of volume one, and actually after
proper scaling (dilation), any Borel probability measure that is
absolutely continuous with respect to Lebesgue measure and that has a
bounded density.

Our main result is as follows.

\begin{theorem}
  \label{thm:main}
Let $C$ be an unconditional convex body in $\R^N$ and $\nu$ be a
radial measure on $\R^n$ of the form $d\nu(x)= \rho(|x|)\, dx $ with
$\rho:[0,+\infty)\to [0 +\infty)$ decreasing.  If $X_1, \ldots, X_N$
    are $N$ independent random vectors in $\R^n$ whose laws are in
    $\mathcal P_n$, then
 \begin{equation}\label{mainexp}
    \EE\big[\nu\big(([X_1\cdots X_N]C)^{\circ}\big) \big] \leq
   \EE\big[\nu\big(([Z_1\cdots Z_N]C)^{\circ}\big) \big] .
  \end{equation}
   where $Z_1, \ldots,Z_N$ are independent random vectors uniformly
   distributed in the Euclidean ball $D_n\subset \R^n$ of volume one.

Moreover, if $\rho^{-1/(n+1)}:[0,+\infty)\to [0, +\infty]$ is convex,
  then, with the same notation, we also have that
 \begin{equation}\label{maindist}
    \forall t> 0, \quad \PP \big[\nu\big(([X_1\cdots
        X_N]C)^{\circ}\big)\ge t\big] \leq \PP
    \big[\nu\big(([Z_1\cdots Z_N]C)^{\circ}\big)\ge t\big].
  \end{equation}
\end{theorem}

Throughout the text, we will use the terms increasing and decreasing
in the non-strict sense.

In particular, note that both results~\eqref{mainexp}
and~\eqref{maindist} hold in the case where $\nu=\abs{\cdot}$ is
Lebesgue measure, and more generally when $\nu(\cdot)=|\cdot \, \cap\,
rB_2^n|$ is the restriction of Lebesgue measure to any Euclidean ball
$\{|x|\le r\}$, $r\in (0,+\infty]$. Of course,~\eqref{mainexp}
formally follows from~\eqref{maindist}, but we stated it first because
we can prove it for a more general class of measures $\nu$. We do not
claim, though, that the further convexity assumption is necessary for
the inequality to hold in distribution; but it is needed in our proof.

To recover Theorem~\ref{thm:random-santalo} from the previous theorem,
simply notice that if $C = B_1^N$, the unit ball in $\ell_1^N$, then
\begin{equation*}
  [X_1\cdots X_N] B_1^N =\conv\{\pm X_1\cdots \pm X_N\}.
\end{equation*}
More generally, using $C=B_q^N$, for some $1\le q\le +\infty$, we also
recover in \S 5 below an inequality which implies the polar
$L^p$ centroid body inequality of Lutwak-Zhang \cite{LZ_BS}.

In \S 2 we recall some basic notation and facts from convex
geometry, in particular Borell's terminology and results concerning
dimensional forms of Pr\'ekopa's theorem. The content of \S 3
might be of independent interest. There, we first recall and extend
Busemann type results, for which we explain how to derive them from
the aforementioned dimensional inequalities by a simple change of
variable. Then we apply these results to the measure of the polar
sets along generalized shadow systems, in the spirit of the work of
Campi and Gronchi~\cite{CG_product}. In \S 4, we start by
recalling the rearrangement inequality of
Rogers-Brascamp-Lieb-Luttinger type, in the form put forward by
Christ~\cite{Christ}. With these ingredients in hand, we procede, at
the end of \S 4, to give the proof of our main result,
Theorem~\ref{thm:main}. Finally, \S 5 presents some further
applications of our result to (non-random) geometric inequalities.

\section{Preliminaries}
We work in Euclidean spaces $\R^n, \R^N\subset
\R^{n+N}=\R^n\oplus\R^N$ with the canonical embeddings and we assume
that $N,n\gr 1$. The usual inner-product is denoted $\langle \cdot,
\cdot \rangle$ with associated Euclidean norm $\abs{\cdot}$ and
the standard unit vector basis is $e_1,\ldots,e_{n+N}$.  We also use
$\abs{\cdot}$ for the $n$-dimensional Lebesgue measure and the
absolute value of a scalar, the use of which will be clear from the
context. The Euclidean ball of radius one is denoted $B_2^n$ and its
volume $\omega_n:=\abs{B_2^n}$.  We reserve $D_n$ for the Euclidean
ball of volume one, i.e., $D_n=\omega_n^{-1/n} B_2^n$; Lebesgue
measure restricted to $D_n$ is $\lambda_{D_n}$; it belongs to
$\mathcal P_n$.  The unit sphere is $S^{n-1}\subset \R^n$ and is
equipped with the Haar probability measure $\sigma$, which is the
usual rotationally invariant measure on $S^{n-1}$, normalized to be a
probability measure. Recall that $B_p^n$ denotes the unit-ball of
$\ell_p^n$, $1\ls p\ls \infty$.

The support function of a convex set $K\subset \R^n$ is given by
\begin{equation*}
  h_K(x) =\sup\{\langle y, x\rangle\; : \ y\in K\} \quad (x\in \R^n).
\end{equation*}
If $K$ and $L$ are convex sets in $\R^n$ then
\begin{equation*}
  h_K(x)+h_L(x) = h_{K+L}(x) \quad (x\in \R^n),
\end{equation*} where $K+L$ is the Minkowski sum of $K$ and $L$:
\begin{equation*}
  K+L =\{ k + l\; : \  k\in K, l\in L\}.
\end{equation*}
If $K\subset \R^n$ is a convex set, the polar $K^{\circ}$ of $K$
is defined by $K^{\circ}=\{y\in \R^n\; : \ h_K(y)\ls 1\}$.  A \emph{convex body}
$K\subset \R^n$ is a compact, convex set with non-empty interior. 
A set $K\subset\R^n$ is a \emph{star-body} with respect to the origin if it is compact, with the
origin in its interior and for every $x\in K$ and $\lambda\in[0,1]$ we have $\lambda x\in K$. Then its \emph{gauge function} is denoted by
$\|\cdot\|_K$ and is defined as $\|x\|_K=\inf\{t>0 : x\in tK\}$. 
We
say that $K$ is (origin) symmetric if $K=-K$. We already gave the definition of \emph{unconditional convex bodies} which are an important subfamily of symmetric convex bodies.

For $K,L$ being convex sets in $\R^n$, we let $\delta^{H}(K,L)$ denote the Hausdorff
distance between them, i.e.,
\begin{equation*}
  \delta^{H}(K,L) = \inf\{\eps>0\; : \ K\subset L+\eps B_2^n, L\subset
  K+\eps B_2^n\};
\end{equation*}
or equivalently, in terms of support functions, 
\begin{equation*}
  \delta^{H}(K,L) = \sup_{\theta \in S^{n-1}}\abs{h_K(\theta)-h_L(\theta)}.
\end{equation*}

Let $\mathcal{K}^n_{\circ}$ denote the class of all convex bodies that
contain the origin in their interior. We will make use of the following basic facts (see, e.g.
\cite{Moszynska, Schneider_book}).

\begin{lemma}
  \label{lemma:Hausdorff}
  Let $K,L,K_1,K_2,\ldots\in \mathcal{K}^n_{\circ}$ be such that
  $K_N{\overset{\delta_H}{\longrightarrow}}K$ as $N\rightarrow
  \infty$.  Then
  \begin{itemize}
    \item[(i)]$K_N^\circ\ {\overset{\delta_H}{\longrightarrow}}\ K^\circ$
      as $N\rightarrow \infty$
    \item[(ii)]$K_N\cap L\ {\overset{\delta_H}{\longrightarrow}}\ K\cap L$
      as $N\rightarrow \infty$
    \item[(iii)]$K_N+L\ {\overset{\delta_H}{\longrightarrow}}\ K+L$ as $N\rightarrow
      \infty$
  \end{itemize}
\end{lemma}

If $A\subset \R^n$ is a Borel set with finite volume, the symmetric
rearrangement $A^{\ast }$ of $A$ is the (open) Euclidean ball centered
at the origin whose volume is equal to that of $A$.  The symmetric
decreasing rearrangement of $1_A$ is defined by
$(1_A)^{\ast}:=1_{A^{\ast }}$. If $f:{\R}^n\rightarrow {\R}^+$ is
an integrable function, we define its symmetric decreasing
rearrangement $f^{\ast}$ by
\begin{equation*}
  f^{\ast }(x)=\int_0^{\infty }1^{\ast }_{\{ f> t\}}(x)dt
  =\int_0^{\infty }1_{\{ f>t\}^{\ast }}(x)dt.
\end{equation*}
The latter should be compared with the ``layer-cake representation''
of $f$:
\begin{equation}
  \label{eqn:layer_cake}
  f(x)=\int_0^{\infty }1_{\{ f> t\}}(x)dt;
\end{equation}
see \cite[Theorem 1.13]{LL_book}.  The function $f^{\ast}$ is
radially-symmetric, decreasing and equimeasurable with $f$, i.e.,
$\{f>\alpha\}$ and $\{f^*>\alpha\}$ have the same volume for each
$\alpha > 0$.  By equimeasurability one has $\norm{f}_p=\norm{f^*}_p$
for each $1\ls p\ls \infty$, where $\norm{\cdot}_p$ denotes the
$L_p(\R^n)$-norm.  If $\mu\in \mathcal{P}_n$ has density $f_\mu$, we
let $\mu^*$ denote the measure in $\mathcal{P}_n$ with density
$f_{\mu}^*$.  For completeness, recall that for a nonnegative function
$f$ in $\R^n$, the rearrangement $f^\ast$ can be reached by a sequence
of \emph{Steiner symmetrizations} $f^*(\cdot|\theta)$, which
correspond to symmetrization in dimension one in the direction
$\theta\in S^{n-1}$; namely $f^*(\cdot|\theta)$ is obtained by
rearranging $f$ (in dimension $1$) along every line parallel to
$\theta$.  The function $f^*(\cdot|\theta)$ is symmetric with respect
to $\theta^\perp$ (by this we mean invariant under the hyperplane
reflection $\sigma_\theta(x) := x - 2 \langle x, \theta \rangle
\theta$).  We refer the reader to the book \cite{LL_book} for further
background material on rearrangements of functions.

\medskip

Let us now recall the results and terminology of
Borell~\cite{borell74, borell75}.
\begin{definition}[Borell's terminology]
Let $s \in[-\infty, 1]$. A Borel measure $\mu$ on $\R^n$ is called \emph{$s$-concave}  if
\[
\mu \left( (1-\la) A + \la B \right) \ge \left( (1-\la) \mu(A)^s + \la \mu(B)^s\right)^{1/s}
\] 
for all compact sets $A, B \subset \R^n$ such that $\mu(A) \mu(B) > 0$.
For $s=0$, one says that $\mu$ is $\log$-concave and the inequality is read as 
$
\mu \left( (1-\la) A + \la B \right) 
\ge 
 \mu(A)^{1-\la} \mu(B)^{\la}.
$
For $s=-\infty$,  the measure is said to be \emph{convex} and the inequality is replaced by 
\[
\mu \left( (1-\la) A + \la B \right) 
\ge 
\min\left(\mu(A),\mu(B)
\right).
\] 
\end{definition}

Notice that the class of $s$-concave measures on $\R^n$ is decreasing in $s$, so that convex measures form the largest one.

We have an analogous notion of $\gamma$-concavity for functions. Namely, by definition, a nonnegative, non-identically zero, function $\psi$  is  \emph{$\gamma$-concave} if: $(i)$ for $\gamma>0$, $\psi^\gamma$ is concave on $\{\psi>0\}$; $(ii)$ for $\gamma=0$, $\log\psi$ is concave on $\{\psi>0\}$; $(iii)$ for $\gamma<0$,  $\psi^\gamma$ is convex on $\{\psi>0\}$.

 In \cite{borell74,borell75}, Borell established the following
 complete characterization of convex measures.  An $s$-concave measure
 $\mu$ is always supported on some convex subset of an affine subspace
 $E$ where it has a density. Moreover, if $\mu$ is a measure on $\R^n$
 absolutely continuous with respect to Lebesgue measure with density
 $\psi$, then it is $s$-concave if and only if its density $\psi$ is
 $\gamma$-concave where $\gamma=s/(1-ns)$. In particular, a measure
 $\mu$ with density $\psi$ on $\R^n$ is a convex measure if and only
 if $\psi$ is $-1/n$-concave. A crucial tool in our arguments will be
 the dimensional form of Pr\'ekopa's theorem obtained
 in~\cite{borell75, BL} as a corollary of the functional versions of
 the Brunn-Minkowski inequality, known as the
 \emph{Borell-Brascamp-Lieb inequalities}. It can also be seen as a
 direct consequence of the aforementioned characterization of Borell
 and the fact that the marginals of an $s$-concave measure are always
 $s$-concave. Thus the following theorem could be called by many names
 such as "Borell-Brascamp-Lieb restricted to convex functions" or "the
 dimensional Pr\'{e}kopa's theorem" or "the functional Brunn's
 principle". In this paper, we shall use the last two names.

\begin{theorem}{\rm(Functional Brunn's principle)}\label{extensionofPrekopa}
Let $\varphi: \R^{n+1} \to (0,\infty]$ be a positive convex function
  and let $\alpha>0$. Then the function $\Phi$ defined on $\R$ by
$$ \Phi(t)= \bigg{(}\int_{\R^n}\varphi(t,x)^{-n-\alpha}\, dx\bigg{)}^{-\frac{1}{\alpha}}$$
is convex.
\end{theorem} 

We shall also sometimes combine it with the following easy and
well-known lemma.

\begin{lemma}{\label{homog}}
Let $f:\R^n\to\R_+$ be a convex function. Then the function $\varphi:
\R^n\times(0,+\infty) \to \R_+$ defined for
$(z,s)\in\R^n\times(0,+\infty)$ by $\varphi(z,s)=s f(z/s)$ is convex.
\end{lemma}

\begin{proof}
Notice that $\varphi$ is positively homogeneous in the sense that
$\varphi(\lambda z, \lambda s)=\lambda\varphi(z,s)$ for every
$z\in\R^n$ and $s,\lambda>0$. For every $s_1,s_2>0$,
$\la_1,\la_2\ge0$, with $\la_1+\la_2=1$ and $z_1,z_2\in\R^{n}$, with
$f(z_1/s_1)>0$ and $f(z_2/s_2)>0$ one has
\begin{eqnarray*}
\varphi(\la_1z_1+\la_2z_2,\la_1s_1+\la_2s_2)&=&(\la_1s_1+\la_2s_2)f\left(\frac{\la_1s_1\frac{z_1}{s_1}+\la_2s_2\frac{z_2}{s_2}}{\la_1s_1+\la_2s_2}\right)\\ &\le&\la_1s_1f\left(\frac{z_1}{s_1}\right)+\la_2s_2f\left(\frac{z_2}{s_2}\right)\\ &=&\la_1\varphi(z_1,s_1)+\la_2\varphi(z_2,s_2).
\end{eqnarray*}
\end{proof}

\section{Busemann's theorem and shadow systems  for convex measures}

In this section, we first recall a generalization of Busemann's
inequality, that we derive from the functional Brunn's Principle
(Theorem~\ref{extensionofPrekopa}) by an elementary argument; we then
use it to deduce an extension of Busemann's theorem to convex
measures. In the second part of this section, we combine these
inequalities to extend a theorem of Campi and
Gronchi~\cite{CG_product} to generalized shadow systems and convex
measures.

\subsection{Busemann's theorem for convex measures}

The following theorem is Bobkov's generalization \cite{bobkov} of a
theorem due to Ball \cite{ball} in the log-concave case. The short
proof that we give below shows that it follows from the functional
Brunn's principle by a change of variable (which simplifies the
argument given by Bobkov \cite{bobkov}). This theorem enables one to
attach to a function with some concavity properties a family of convex
bodies (the unit balls of the gauges given by the theorem), sometimes
called Ball's bodies. It is a key technique due to Ball to extend
results from convex sets to log-concave functions.

\begin{theorem}[\cite{ball,bobkov}]
\label{ball-bobkov}
Let $p>0$ and $f:\R^n\to\R_+$ be $\gamma$-concave for some $\gamma\ge
-1/(p+1)$. Then the function $F:\R^n\to\R_+$ defined by
\[F(x)=\left(\int_0^{+\infty}f(rx)\, r^{p-1}dr\right)^{-\frac{1}{p}}\]
is a gauge on $\R^n$.
\end{theorem}

\begin{proof}
Let us denote by $\varphi$ the convex function such that
$\varphi=f^{-1/(p+1)}$. Using the change of variable $r=1/s$, we get
\[
\int_0^{+\infty}f(rx)r^{p-1}dr=\int_0^{+\infty}\left(s\varphi(x/s)\right)^{-(p+1)}ds
\]
From Lemma \ref{homog}, we know that $(x,s)\to s\varphi(x/s)$ is
convex on $(0,+\infty)\times\R^n$. From the functional Brunn's
principle (Theorem~\ref{extensionofPrekopa}), we conclude that $F$ is
convex.
\end{proof}

The previous proof clarifies the relation between Brunn-Minkowski type
results and Busemann's theorem, which follows from
Theorem~\ref{ball-bobkov} for $p=1$, as we shall see below. Indeed,
even if we are interested in the case of a log-concave function $f$
(which means that $f$ is $0$-concave, and therefore also
$-1/2$-concave), the shortest proof uses the dimensional Pr\'ekopa
theorem for $-1/2$-concave densities (rather than the usual Pr\'ekopa
theorem for log-concave densities).

Actually, combining the previous theorem with one more instance of the
functional Brunn's principle, one may also extend Milman-Pajor's
generalization of Busemann's theorem \cite{MP} to densities with less
concavity.

\begin{proposition}\label{busemann-measures-2}
Let $E$ be a $k$-dimensional subspace of $\R^n$ and $p>0$. If
$\varphi:\R^n \to \R_+$ is a $\gamma$-concave function for some
$\gamma\ge -1/(k+p+1)$, then the function $\Phi: E^\bot\to \R_+$
defined for $v\in E^\bot\setminus \{0\}$ by
$$ \Phi(v)=\abs{v}^{2p-1\over p}\left(\int_{E\oplus\R_+ v} \langle x,
v \rangle^{p-1}\, \varphi(x) \, dx \right)^{-1/p}
$$
 is a gauge on $E^\bot$. 
\end{proposition}

This result formally contains Theorem~\ref{ball-bobkov} (the case
$E=\{0\}$, $k=0$), but in the application below we will rather use the
cases $k=n-2$ or $k=(n+1)-2=n-1$ and $p=1$.

\begin{proof}
The homogeneity is clear so it suffices to prove the
convexity. Introduce the function $f:E^\bot \to \R_+$ defined by
$$f(x)=\int_E \varphi(x+y)dy.$$ From the functional Brunn's principle
(Theorem~\ref{extensionofPrekopa}), it follows that $f$ is
$-1/(p+1)$-concave. By Fubini and the normal parametrization of $\R_+
v$ by $s \frac{v}{|v|}$ we have that, for $v\neq 0$,
 \begin{eqnarray*}
 \Phi(v)&=&\abs{v}^{2p-1\over p}\left(\int_{\R_+ v} \langle x, v
 \rangle^{p-1}\, f(x)\, dx \right)^{-1/p} \\ & = & \abs{v}^{2p-1\over
   p}\left(\int_{0}^{+\infty} |v|^{p-1} s^{p-1} \, f\big( s
 \frac{v}{|v|} \big)\, ds \right)^{-1/p} \\ & =&
 \left(\int_0^{+\infty}r^{p-1} \, f(rv) \, dr\right)^{-1/p}.
 \end{eqnarray*}
 So the result follows from Theorem \ref{ball-bobkov} with $E^\bot$ in
 place of $\R^n$.
\end{proof}

We now apply the previous results to extend Busemann's theorem
\cite{busemann} to convex measures. The case of log-concave measures
is due to Kim, Yaskin and Zvavitch \cite{KYZ} who proved it somewhat
differently by applying the usual Busemann's theorem to Ball's body
associated to the measure. The same method could also be used in our
more general setting.  

For any measure $\nu$ on $\R^n$ with a density
$\psi$, $d\nu(x) = \psi(x)\, dx$, and for every hyperplane $H$ we
define the $\nu$-measure of $H$ to be
\[
\nu^+(H)=\int_H\psi(x)dx,
\]
where $dx$ denotes Lebesgue measure on $H$.

\begin{theorem}[Busemann's theorem for convex measures]\label{busemann-measures}
Let $\nu$ be a convex measure with even density $\psi$ on $\R^n$. Then
the function $\Phi$ defined on $\R^n$ by $\Phi(0)=0$ and for $z\neq 0$
\[
\Phi(z)=\frac{|z|}{\nu^+(z^\bot)}=\frac{|z|}{\int_{z^\bot}\psi(x)dx}
\]
is a norm. 
\end{theorem}

\begin{proof}
The homogeneity and symmetry are clear so it suffices to prove the
convexity. This is equivalent to proving that the restriction of
$\Phi$ to any linear $2$-dimensional subspace is convex.

So let $F$ be a $2$-dimensional subspace of $\R^n$ and set $E=
F^\perp$. Introduce the rotation $R$ of angle $\pi/2$ in the plane
$F$. Then for all $z\in F$, $z\neq 0$,
$$ \int_{(Rz)^\bot}\psi = \int_{E\oplus \R z}\psi = 2 \int_{E\oplus
  \R_+ z}\psi .$$ We now apply Proposition~\ref{busemann-measures-2}
with $k=n-2$, $p=1$, and $\varphi=\psi$, which is $-1/n=
-1/(k+p+1)$-concave by assumption,. It gives exactly that $z\to
\Phi(Rz)$ is convex. Since $R$ is linear, the convexity of $\Phi$
follows.
\end{proof}


\begin{remark} 
Since the restriction of a convex measure to a convex set $K$ with
non-empty interior remains a convex measure, the theorem also implies
that the function
\[
z\to \frac{|z|}{\nu^+(K\cap z^\bot)}=\frac{|z|}{\int_{K\cap
    z^\bot}\psi(z)dz}
\]
is a norm. 
\end{remark}

\subsection{Campi-Gronchi type results for convex measures}

We now generalize a theorem of Campi and Gronchi \cite{CG_product} on
shadow systems. The inspiring argument of Campi and Gronchi relies on
the formula $|K|=\omega_n\int_{S^{n-1}} h_K^{-n}\, d\sigma$ (followed
by a stereographic projection) and on the dimensional Pr\'ekopa
inequality. Our more general situation requires a slightly different
look at the problem and it turns out that the Busemann theorem for
convex measures (Theorem~\ref{busemann-measures}) is a good tool to
work with.


Shadow systems were defined by Shephard \cite{Shephard_shadow} in the
following way. Let $C$ be a closed convex set in $\R^{n+1}$. Let
$(e_1, \cdots, e_{n+1})$ be an orthonormal basis of $\R^{n+1}$, we
write $\R^{n+1}=\R^n\oplus\R e_{n+1}$, so that $\R^n=e_{n+1}^\bot$
. Let $\theta\in \R^n$, with $|\theta|=1$. For every $t\in\R$ let
$P_t$ be the projection onto $\R^n$ parallel to $e_{n+1}-t\theta$: for
$x\in \R^n$ and $s\in \R$,
$$P_t(x+se_{n+1})= x + t s \theta.$$ We denote $K_t=P_t C\subset
\R^n$.  Then the family $(K_t)$ is a shadow system of convex sets. The
next theorem extends the Campi-Gronchi result \cite{CG_product},
originally proved for Lebesgue measure, to the setting of convex
measures.

\begin{theorem}\label{campi-gronchi-measures}
Let $\nu$ be a measure on $\R^n$ with a density $\psi$ which is even
and $\gamma$-concave on $\R^n$ for some $\gamma\ge -1/(n+1)$. Let
$(K_t)$ be a shadow system of centrally symmetric convex sets. Then
the function $t\to \nu(K_t^\circ)^{-1}$ is convex.
\end{theorem}

\begin{proof} Write
\begin{eqnarray*}
K_t^\circ &=& \{x \in \R^n \, : \ \langle x, P_t y\rangle \le 1 ,
\ \forall y \in C\} = \{x \in \R^n \, : P_t^\ast x \in C^\circ \}
\end{eqnarray*}
where $P_t^\ast$ is the adjoint of $P_t$. Observe that $P_t^\ast$ is
the projection on $(e_{n+1}-t\theta)^\bot$ parallel to $e_{n+1}$ and
that $P P_t^\ast = P$, where $P$ denotes the orthogonal projection on
$\R^n=e_{n+1}^\bot$. It follows that
\[
K_t^\circ=P(C^\circ\cap (e_{n+1}-t\theta)^\bot), 
\]
We now perform the change of variables $x=P(y)$ which is a
diffeomorphism from $(e_{n+1}-t\theta)^\bot$ onto $\R^n$ with Jacobian
equal to $1/\sqrt{1+t^2}$. We get
\[
\nu(K_t^\circ)=\int_{P(C^\circ\cap
  (e_{n+1}-t\theta)^\bot)}\psi(x)dx=\int_{C^\circ\cap
  (e_{n+1}-t\theta)^\bot}\psi(P(y))\frac{dy}{\sqrt{1+t^2}}.
\]
Since $K_t$ is symmetric, so is $C^\circ\cap (e_{n+1}-t\theta)^\bot$.
Thus, the function $\varphi$ defined on $\R^{n+1}$
by $$\varphi(y)=1_{C^\circ(y)}\, \psi(P(y))$$ is $-1/(n+1)$-concave on
$\R^{n+1}$ and its restriction to $(e_{n+1}-t\theta)^\bot$ is even. It
follows that the measure $\nu$ with density $\varphi$ on $\R^{n+1}$ is
a convex measure. Since
\[
\nu(K_t^\circ)=\frac{\int_{(e_{n+1}-t\theta)^\bot}\varphi(y)dy}{\sqrt{1+t^2}}=\frac{\nu^+((e_{n+1}-t\theta)^\bot)}{|e_{n+1}-t\theta|},
\]
we conclude from Busemann's theorem for measures
(Theorem~\ref{busemann-measures} with $n+1$ in place of $n$) that the
function $t\to \nu(K_t^\circ)^{-1}$ is convex.
\end{proof}

\begin{remark}
From Theorem \ref{campi-gronchi-measures}, one easily deduces a new
and simple proof of the following result established by Meyer-Reisner
\cite{MR} in the log-concave case and generalized by Bobkov
\cite{bobkov}. For $a\in\R^n$, denote $B(a)=\{x\in\R^n : |\langle
x,a\rangle|\le 1\}$. Let $\nu$ be a measure on $\R^n$ with a density
$\psi$ which is even and $\gamma$-concave on $\R^n$ for some
$\gamma\ge -1/(n+1)$. Then the function $W(a)=\nu(B(a))^{-1}$ is
convex on $\R^n$.

Indeed, we can use the simplest shadow system: take $C=[-(a+e_{n+1}),
  a+e_{n+1}]$ and $\theta\in S^{n-1}$. One gets $K_t=[-(a+t\theta),
  a+t\theta]$ and $K_t^\circ=B(a+t\theta)$. Thus
Theorem~\ref{campi-gronchi-measures} implies that $t\to
W(a+t\theta)$ is convex for all $a\in\R^n$ and $\theta\in S^{n-1}$,
which implies that $W$ is convex on $\R^n$.
\end{remark}

Arguments similar to those in the proof of Theorem
\ref{campi-gronchi-measures} also apply to the following
generalization of shadow systems.

\begin{proposition}\label{prop:gene_shadow}
Let $n,N$ be positive integers and $\mathcal{C}$ be a centrally
symmetric closed convex set in $\R^n\times\R^N$. Let $\theta\in
S^{n-1}$. For $t\in\R^N$ and $(x,y)\in\R^n\times\R^N$, we define
$P_t(x,y)=x+\langle y,t\rangle\theta$ and $K_t=P_t(\mathcal{C})$. Let
$\nu$ be a measure on $\R^n$ with a density $\psi$ with respect to
Lebesgue measure that is even and $-1/(n+1)$-concave on $\R^n$. Then
\begin{itemize}
\item[i)] $t\to \nu(K_t^\circ)^{-1}$ is convex on $\R^N$.
\item[ii)] if $\mathcal{C}$ and $\psi$ are symmetric with respect to
  $\theta^\bot$ then $t\to\nu((K_t)^\circ)^{-1}$ is even and convex on
  $\R^N$.
\end{itemize}
\end{proposition}

\begin{proof} (i) The proof follows that of Theorem~\ref{campi-gronchi-measures}. Again, the result relies on a proper application of Busemann's theorem for measures (as before in dimension $n+1$ for $n$-dimensional sections); actually, it will be more handy, but equivalent,  to go back to the formulation of Proposition~\ref{busemann-measures-2} rather than to quote Theorem~\ref{busemann-measures}.

We work on $\R^n \oplus \R^N$.  For $t=(t_1,\dots,t_N)\in\R^N$, the
linear map $P_t$ in the proposition is the projection onto $\R^n$
parallel to the $N$-dimensional subspace $\span(e_{n+1}-t_1\theta,
\dots, e_{n+N}-t_N\theta)$. Introducing the $(n-1)$ dimensional
subspace $E=\theta^\bot\cap\R^n$, we observe that $P_t^*$ is the
projection onto the $n$-dimensional space $E\oplus\R(\theta
+\sum_{i=1}^Nt_ie_{n+i})$ parallel to $\R^N$. Let us denote by $P$ the
orthogonal projection onto $\R^n$. Then one has that
$$K_t=P(C^\circ\cap(E\oplus\R(\theta + \sum_{i=1}^Nt_ie_{n+i}))).$$
The projection $P$ induces a diffeomorphism from $E\oplus\R(\theta +
\sum_{i=1}^Nt_ie_{n+i})$ to $\R^n$ with Jacobian $1/\sqrt{1+
  t_1^2+\cdots+t_N^2}$. Therefore we get
\[
\nu(K_t^\circ)=\int_{C^\circ\cap(E\oplus\R(\theta +
  \sum_{i=1}^Nt_ie_{n+i}))}\frac{\psi(P(y))}{\sqrt{1+
    t_1^2+\cdots+t_N^2}}dy.
\]
The function $\varphi(y)= 1_{C^\circ}(y)\psi(P(y))$ is
$-1/(n+1)$-concave on $\R^n \simeq E\oplus\R(\theta +
\sum_{i=1}^Nt_ie_{n+i})$. Using Proposition~\ref{busemann-measures-2}
for $\varphi$, $p=1$ and $k=n-1$ we can conclude, after composition
with the linear map $t\to v=\theta + \sum_{i=1}^Nt_ie_{n+i}$, that the
function $t\to \nu((K_t)^\circ)^{-1}$ is convex.

(ii) Let us denote by $\sigma_\theta$ the orthogonal symmetry with
respect to $\theta^\bot$ in $\R^n\times\R^N$. Since $P_{-t}=
\sigma_\theta\circ P_t\circ\sigma_\theta$, we deduce that
\[
K_{-t}=P_{-t}\mathcal{C}=\sigma_\theta\circ P_t\circ\sigma_\theta
\mathcal{C}=\sigma_\theta\circ P_t \mathcal{C}=\sigma_\theta K_t.
\]
Therefore, using $\psi\circ\sigma_\theta=\psi$, we get
$$\nu((K_{-t})^\circ)= \nu((\sigma_\theta K_{t})^\circ)=\nu(\sigma_\theta ((K_{t})^\circ))
=\nu(K_t^\circ).$$ 
We conclude that $t\to \nu((K_t)^\circ)^{-1}$ is even.
\end{proof}

Let us mention that such generalized shadow systems and even more general
notions were considered by Shephard in his seminal article
\cite{Shephard_shadow}. 


Lastly, we state a key corollary that will be used in the proof of
Theorem \ref{thm:main}.

\begin{corollary}
\label{cor:GCCviaCG}
Let $r\ge 0$, $C$ be an origin-symmetric convex set in $\R^N$, let
$\theta\in S^{n-1}$ and $y_1,\dots, y_N\in\theta^\bot$. Let $\nu$ be a
measure on $\R^n$ with a density $\psi$ which is $-1/(n+1)$-concave on
$\R^n$, even and symmetric with respect to $\theta^\bot$. Then, the
map
$$
(t_1,\dots,t_N)\to 
\nu(([y_1+t_1\theta \cdots y_N+t_N\theta]C+r B_2^n)^\circ)^{-1}
$$
is even and convex on $\R^N$.
\end{corollary}

\begin{proof}
Let $\mathcal{C}=[y_1+e_{n+1}\cdots y_N+e_{n+N}]C+r B_2^n$. Then
$\mathcal{C}$ is an origin-symmetric convex set in $\R^n\times\R^N$
which is symmetric with respect to $\theta^\bot$ in $\R^{n+N}$ since
$[y_1+e_{n+1}\cdots y_N+e_{n+N}]C\subset \theta^\bot$. Let $P_t :
\R^n\times\R^N\to\R^n$ be defined as in Proposition
\ref{prop:gene_shadow} and let $K_t=P_t\mathcal{C}$. Then one has
\begin{eqnarray*}
K_t &=& P_t([y_1+e_{n+1}\cdots y_N+e_{n+N}]C+r B_2^n)\\
&=& [P_t(y_1+e_{n+1})\cdots P_t(y_N+e_{n+N})]C+r P_t B_2^n\\
&=& [y_1+t_1\theta \cdots y_N+t_N\theta]C+r B_2^n.
\end{eqnarray*}
By $ii)$ of the
preceding proposition we can conclude.
\end{proof}


\section{Proof of Theorem \ref{thm:main}}

We start by recalling the rearrangement inequalities that are at the
heart of the argument. Recall from Section~\S2 the notation $g^\ast$ for the
radially decreasing rearrangement of a nonnegative function $g$. The
Brascamp-Lieb-Luttinger inequality~\cite{BLL}, which was actually
anticipated by Rogers~\cite{Rogers} as pointed out
in~\cite{WangMadiman}, states that given $k$ (integrable) nonnegative
functions $g_1, \ldots, g_k$ on $\R$ and $N k$ constants
$\{c_{i,j}\}_{i\le k, j\le N}$ we have that
\begin{equation}\label{rbbl}
\int_{\R^N} \prod_{i=1}^k g_i\big(c_{i1} s_1+ \ldots + c_{iN} s_N\big)
\, ds_1 \ldots ds_N \le \int_{\R^N}\prod_{i=1}^k g_i^\ast\big(c_{i1}
s_1 + \ldots + c_{iN} s_N\big) \, ds_1 \ldots ds_N.
\end{equation}

Christ~\cite{Christ} derived a useful consequence of the
Rogers-Brascamp-Lieb-Luttinger inequality. In particular, it was shown
in~\cite{PaoPiv_probtake} that Christ's formulation is very well
adapted to geometric inequalities in convexity, as it provides a handy
interface between generalized Steiner symmetrization (as
in~\eqref{eqn:F_Y} below) and the Rogers-Brascamp-Lieb-Luttinger
inequality (see also~\cite[page 15]{BaeLoss} and \cite[Lemma
  3.3]{GT_radius}). The result is as follows. 
  
  \begin{theorem}[\cite{Christ,PaoPiv_probtake}]\label{thm:GCC->max} 
Let $F:(\R^n)^N=\otimes_{i=1}^N \R^n
\rightarrow \R_{+}$. We have that 
     \begin{multline}
        \label{eqn:GCC->max}
    \int_{(\R^n)^N} F(x_1,\ldots,x_N)\, f_1(x_1)\cdots f_N(x_N) \, dx_1 \ldots dx_N \\
\le 
\int_{(\R^n)^N} F(x_1,\ldots,x_N)\, f_1^\ast (x_1)\cdots f_N^\ast (x_N) \, dx_1 \ldots dx_N
      \end{multline}
holds for any integrable $f_1, \ldots, f_{N}:\R^n\to \R_+$ provided that
$F$ satisfies the following condition: for every $z\in
S^{n-1}\subset \R^n$ and for every $Y=(y_1,\ldots,y_N)\subset
(z^{\perp})^N\subset (\R^n)^N$, the function $F_{z,Y}:\R^N \to \R_+$
defined by
 \begin{equation}
  \label{eqn:F_Y}
  F_{z,Y}(t):=F(y_1+t_1z, \ldots, y_N+t_N z). 
\end{equation} 
 is even and quasi-concave. 
\end{theorem}

Briefly, the argument from~\eqref{rbbl} towards this result,  goes as follows. First, by putting
extra functions of the form $1_{[-r_j, r_j]}$ and using the fact that
a symmetric convex set is the intersection of symmetric strips
$\{|x\cdot a_j|\le r_j\}$, we see that the inequality~\eqref{rbbl}
remains true if we integrate on a symmetric convex subset of
$\R^N$. Then, by the decomposition~\eqref{eqn:layer_cake} of a
function into its level sets, we see that~\eqref{rbbl} remains true if
we integrate against an even quasi-concave function $G$ on $\R^N$. In
particular, we have for $N$ nonnegative functions $g_i$ on $\R$, that
\begin{equation*}\label{rbbl2}
\int_{\R^N} G(s)\prod_{i=1}^N g_i\big(s_i)  \, ds_1 \ldots ds_N \le 
\int_{\R^N} G(s)  \prod_{i=1}^N g_i^\ast\big(s_i\big) \, ds_1 \ldots ds_N.
\end{equation*}
Then, to move from $n=1$ to arbitrary $n\gr 1$, i.e. for functions
$f_i$ on $\R^n$ and integration against $F$ on $\R^{nN}$ as in the theorem, we
approximate the rearranged function $f_i^\ast$ by a suitable sequence
of Steiner symmetrizations $f_i^*(\cdot|\theta)$, $\theta\in
S^{n-1}$. We can use, by Fubini, on $N$ affine lines in $\R^n$ 
parallel to $\theta$, $y_i + \R \theta$, the previous rearrangement inequality with the
$g_i$'s being the restrictions of the $f_i$'s, the
condition~\eqref{eqn:F_Y} guaranteeing exactly that the
restriction $G$ of $F$ is indeed even and quasi-concave on $\R^N$.  We
refer to~\cite{PaoPiv_probtake} for further details.

With this rearrangement inequality in hand, we can put together all
the pieces needed for the proof of Theorem~\ref{thm:main}. Let us
recall the class of measures $\nu$ on $\R^n$ we can work with, namely
the spherically-invariant measures with
\begin{equation}\label{condnu1}
d\nu(x) = \rho(|x|) \, dx \quad \textrm{with } \rho: [0, +\infty) \to [0, +\infty) \textrm{ decreasing},
\end{equation}
together with the sub-class of those of the form
\begin{equation}\label{condnu2}
d\nu(x) = k^{-(n+1)}(|x|) \, dx \quad \textrm{with } k: [0, +\infty) \to [0, +\infty] \textrm{ convex increasing}.
\end{equation}

We will prove the following more general statement
(Theorem~\ref{thm:main} corresponds to $(ii)$ below, with $r=0$) .

\begin{theorem}
  \label{thm:main_Lebesgue} 
   Let $X_1,\ldots, X_N$ be $N$ independent random vectors in $\R^n$
   whose laws are in $\mathcal{P}_n$ and let $r\ge0$.
  \begin{itemize}
  \item[(i)] If $C$ is an origin-symmetric convex body in $\R^N$ and
    $\nu$ a measure on $\R^n$ of the form~\eqref{condnu1}, then
   \begin{equation}
     \label{eqn:main_Lebesgue1-1}
      \EE\big[\nu\big(([X_1\cdots X_N]C+r B_2^n)^{\circ}\big) \big] \leq
   \EE\big[\nu\big(([X_1^*\cdots X_N^*]C +r B_2^n)^{\circ}\big) \big] 
  \end{equation}
    where $X_1^*, \ldots,X_N^*$ are independent random vectors in
    $\R^n$ whose densities are the symmetric decreasing rearrangement
    of the densities of $X_1,\ldots, X_N$.  Moreover if $\nu$ is of
    the form~\eqref{condnu2} we also have that for every $t\ge0$,
  \begin{equation}
    \label{eqn:main_Lebesgue1-2}
    \PP \big[\nu\big(([X_1\cdots X_N]C+r B_2^n)^{\circ}\big)\ge t\big] \leq
    \PP \big[\nu\big(([X_1^*\cdots X_N^*]C +r B_2^n)^{\circ}\big)\ge t\big].
  \end{equation}

\item[(ii)] If $C$ is an unconditional convex body in $\R^N$ and $\nu$
  a measure on $\R^n$ of the form~\eqref{condnu1}, then
 \begin{equation}
     \label{eqn:main_Lebesgue2-1}
      \EE\big[\nu\big(([X_1\cdots X_N]C+r B_2^n)^{\circ}\big) \big] \leq
   \EE\big[\nu\big(([Z_1\cdots Z_N]C +r B_2^n)^{\circ}\big) \big] 
  \end{equation}
where $Z_1, \ldots,Z_N$ are independent random vectors distributed
according to $\lambda_{D_n}$.  Moreover if $\nu$ is of the
form~\eqref{condnu2}, we also have that for every $t\ge0$,
  \begin{equation}
    \label{eqn:main_Lebesgue2-2}
    \PP \big[\nu\big(([X_1\cdots X_N]C+r B_2^n)^{\circ}\big)\ge t\big] \leq
    \PP \big[\nu\big(([Z_1\cdots Z_N]C +r B_2^n)^{\circ}\big)\ge t\big].
  \end{equation}

\end{itemize}
\end{theorem}

\begin{proof} (i) Let $G$ and $F$ be defined on $(\R^n)^N$ by
  \begin{equation}
    \label{eqn:F}
    G(x_1,\ldots,x_N) = \nu(([x_1\cdots x_N]C+r B_2^n)^{\circ})\quad {\rm and}\quad F=1_{\{G>\alpha\}}
  \end{equation}
  Let $\theta\in S^{n-1}$ and $Y=(y_1,\ldots,y_N)\subset
  (\theta^{\perp})^N$ and let $F_{\theta,Y}$ and $G_{\theta,Y}$ be the
  restrictions of $F$ and $G$ as in (\ref{eqn:F_Y}) with $z=\theta$.
  Note that $F_{\theta,Y}=1_{\{G_{\theta,Y}>\alpha\}}$.
  
 Assume first that $\nu$ is of the form~\eqref{condnu2}.  The
 rotational invariance and the convexity assumption on the density
 ensure that the assumptions of Corollary~\ref{cor:GCCviaCG} are
 satisfied. Thus $G_{\theta,Y}^{-1}$ is even and convex on
 $\R^N$. Hence $G_{\theta,Y}$ and therefore $F_{\theta,Y}$ are
 quasi-concave and even.  Thus we can apply Theorem~\ref{thm:GCC->max}
 to the function $F$ and obtain~\eqref{eqn:main_Lebesgue1-2}.

Next, if $\nu$ satisfies the weaker assumption~\eqref{condnu1}, we
start by applying the previous result in the case of Lebesgue measure
restricted to an Euclidean ball of radius $R>0$ (the density
$1_{RB_2^n}$ is $+\infty$-concave and therefore
$-1/(n+1))$-concave). To condense the notation, we will write $[x_i]$
rather than $[x_1\cdots x_N]$. We have
\begin{equation}\label{eqn:balls}
\forall t>0, \quad \Probb{}{\abs{([X_i]C +r B_2^n)^{\circ}\cap
    RB_2^n}>t} \ls \Probb{}{\abs{([X_i^*]C +r B_2^n)^{\circ}\cap
    RB_2^n} > t}.
        \end{equation} With the notation~\eqref{condnu1}, note that
for $t\in (0, \rho(0))$, the set $\{\rho\ge t\}$ is an Euclidean ball,
open or closed, but the difference is of Lebesgue measure zero, so we
will take later closed balls; we denote by $R(t)$ the corresponding
radius.  By Fubini, for any Borel set $A\subset
\R^n$,  we can write
\[
\nu(A)=\int_0^{+\infty}|A\cap\{\rho\ge t\}|\, dt = \int_0^{\rho(0)} |A\cap R(t) B_2^n|\, dt,
\]
which gives
\begin{multline} \label{eqn:trick}
  \EE \nu(([X_i]C+r B_2^n)^{\circ}) = \EE \int_0^{+\infty} |([X_i]C+r
  B_2^n)^{\circ}\cap R(t)B_2^n|dt\\ = \int_0^{+\infty}\int_0^{+\infty}
  \Probb{}{\abs{([X_i]C +r B_2^n)^{\circ}\cap R(t)B_2^n}>\alpha}
  d\alpha dt
\end{multline}
 Thus~\eqref{eqn:main_Lebesgue1-1} follows from~\eqref{eqn:balls}.\\

(ii) After this step, we have arrived to radially decreasing
 probability distributions. It remains to go to the uniform
 distributions on $D_n$, namely to the
 inequalities~\eqref{eqn:main_Lebesgue2-1}
 and~\eqref{eqn:main_Lebesgue2-2} in the case where $C$ is an
 unconditional convex body. Note that in this case, the functions $G$
 and $F$ defined above are coordinate-wise decreasing in the sense
 that
\begin{multline}\label{condF}
\forall x_1, \ldots, x_N\in\R^n, \quad  (0\ls s_i\ls t_i , \ \forall i\ls N) \\
\Longrightarrow F(s_1x_1, \ldots,s_N x_N)\gr F(t_1x_1,\ldots,t_Nx_N).
\end{multline}
This follows from the fact that, for such $s_i$'s and $t_i$'s, the
unconditionality of $C$ implies that, for every
$x_1,\dots,x_N\in\R^n$,
\[
[s_1x_1\cdots s_Nx_N]C\subset [t_1x_1\cdots t_Nx_N]C.
\]
Then we can apply the following fact
from~\cite[Prop. 3.5]{PaoPiv_smallball}, for which we recall a proof
below for completeness.

\begin{lemma}[\cite{PaoPiv_smallball}] \label{fact}
Let $F:(\R^n)^N\to \R^+$ be a function that satisfies the
condition~\eqref{condF}. If $g_1, \ldots, g_N:\R^+\to [0,1]$ are
nonnegative, bounded by $1$, integrable functions with $ \int_{\R^n}
g_i(|x|) \, dx = 1$ for all $i=1, \ldots N$, then
\begin{multline*}
\int_{(\R^n)^N} F(x_1, \ldots, x_N) \prod_{i=1}^N g_i(|x_i|)\, dx_1 \ldots dx_N \\
\le 
\int_{(\R^n)^N} F(x_1, \ldots, x_N) \prod_{i=1}^N 1_{[0, r_n]}(|x_i|)\, dx_1 \ldots dx_N 
\end{multline*}
where $r_n$ is the radius of $D_n$.
\end{lemma}
We can therefore apply this lemma in the case where $g_i$ is the law
of $X_i^\ast$ which satisfies the assumptions (the fact that the
density of $X_i$ is bounded by $1$ implies indeed that its radial
rearrangement is also bounded by $1$) and to our function
$F$~\eqref{eqn:F} which now satisfies~\eqref{condF}. This yields
$$
   \PP \big[\nu\big(([X_i^\ast]C+r B_2^n)^{\circ}\big)\ge t\big] \leq
    \PP \big[\nu\big(([Z_i]C +r B_2^n)^{\circ}\big)\ge t\big].
    $$ Combined with~\eqref{eqn:main_Lebesgue1-2}, we arrive
    at~\eqref{eqn:main_Lebesgue2-2}. Note that in the proof of (ii)
    we have not exploited the fact that the $g_i$'s are radially
    decreasing, only the fact that they are radial.

To get~\eqref{eqn:main_Lebesgue2-1} for the larger class of measure
$\nu$, we can either use~\eqref{eqn:main_Lebesgue1-1} and the fact
above applied to the function $G$, or else deduce it
from~\eqref{eqn:main_Lebesgue2-2} with the same
trick~\eqref{eqn:trick} as above.

\end{proof}

\begin{proof}[Proof of Lemma~\ref{fact}]
Using Fubini, we see that it is enough to treat each coordinate one
after the other, and so the fact boils down to the following $N=1$
dimensional statement:
$$\int_{\R^n} F(x) g(|x|) \, dx \le \int_{\R^n} F(x)
1_{[0,r_n]}(|x|)\, dx $$ when $g$ is a nonnegative function bounded by
$1$ with $\int_{\R^n} g(|x|)\, dx = 1$, and $F$ is an even function on
$\R^n$ satisfying $F(s x) \ge F(x)$ for all $x\in \R^n$ and $s\in
[0,1]$. This property of $F$ implies that the function $r\to F(r
x_0)$ is decreasing on $\R^+$ for any fixed $x_0\in \R^n$.  Therefore,
by integration in polar coordinates, we see that it suffices to prove
that
$$\int_{0}^{+\infty} f(r) g(r) \, r^{n-1}\, dr \le\int_0^{r_n} f(r) \,
r^{n-1}\, dr $$ when $f$ is a decreasing function and $g$ has values
in $[0,1]$ with $\int_{0}^{+\infty} g(r) \, r^{n-1}\, dr=\int_0^{r_n}
r^{n-1}\, dr $. This is now standard. Denote $\alpha(r):=
(1_{[0,r_n]}(r) - g(r))r^{n-1}$, and observe that
$$\int_0^{+\infty} f(r)\, \alpha(r) \, dr = \int_0^{+\infty}
(f(r)-f(r_n))\, \alpha(r) \, dr \ge 0$$ since the integrand in the
second integral is point-wise nonnegative.
\end{proof}

\section{Applications}
\label{section:apps}  

Here we present some applications of our random theorems to
deterministic geometric inequalities using the law of large
numbers. In particular, we give a more rigorous argument than that
sketched in the introduction on how to recover Blaschke-Santal\'o type
inequalities.

The following result shows that we can pass to the limit in our main
statement when there is almost-sure convergence in the Hausdorff
metric.

\begin{theorem}
  \label{thm:apps}
  Let $(X_i)$ and $(Z_i)$ be sequences of independent random vectors
  in $\R^n$ with each $X_i$ distributed according to the same fixed
  $\mu\in \mathcal{P}_n$ and each $Z_i$ according to $\lambda_{D_n}$.
  Assume that $C_N,C_{N+1},\ldots$ are unconditional convex bodies
  with $C_N\subset \R^N$, $N=n,n+1,\ldots$, such that
  \begin{equation}
    [X_1\cdots X_N]C_N \text{ converges } \otimes_{i=1}^{\infty}\mu
    \text{-a.s. in $\delta^H$}
  \end{equation} and 
  \begin{equation}
    [Z_1\cdots Z_N]C_N \text{ converges }
    \otimes_{i=1}^{\infty}\lambda_{D_n}\text{-a.s. in $\delta^H$}.
  \end{equation} Then, if $\nu$ is a measure 
  on $\R^n$ with a spherically-symmetric, decreasing density, we have
  \begin{equation}
    \label{eqn:main_limit}
    \EE \nu\left( \left(\lim_{N\rightarrow
      \infty}[X_1\cdots X_N]C_N\right)^{\circ}\right) \ls
    \EE \nu\left( \left(\lim_{N\rightarrow
      \infty}[Z_1\cdots Z_N]C_N\right)^{\circ}\right).
  \end{equation}
\end{theorem}

To prove the theorem, we will need the following lemma.

\begin{lemma}
  \label{lemma:converge}
  Let $\nu$ be a measure 
  on $\R^n$ with a spherically-symmetric, decreasing density. Then $\nu$ is
  continuous on $\mathcal{K}^n_{\circ}$ with respect to $\delta^H$.
\end{lemma}

\begin{proof} 
We can restrict ourselves to continuity for sets included in some
compact set. Then by uniform approximation, we may assume that the
density $f_{\nu}=\frac{d\nu}{dx}$ of $\nu$ is of the form:
  \begin{equation*}
    f_{\nu}(x) = \sum_{j=1}^M a_j {1}_{r_j B_2^n}(x) \quad
    (x\in \R^n)
  \end{equation*}where $a_j>0$, $j=1,\ldots,M$ and $r_1> r_2> \ldots> r_M>0$.
  Suppose now that $K,K_1,K_2,\ldots\in \mathcal{K}^n_{\circ}$ and
  $\delta^H(K_N,K)\rightarrow 0$ as $N\rightarrow \infty$. Then, as
  $N\rightarrow \infty$,
  \begin{equation}
    \nu(K_N) = \sum_{j=1}^M a_j \abs{K_N\cap(r_j B_2^n)}
    \rightarrow \sum_{j=1}^M a_j \abs{K\cap(r_j B_2^n)}
    = \nu(K)
  \end{equation}by Lemma \ref{lemma:Hausdorff}.
\end{proof}

\begin{proof}[Proof of Theorem \ref{thm:apps}]
  Let $\eps >0$.  Note that $$[X_1\cdots X_N]C_N+\eps B_2^n\supseteq
  \eps B_2^n,$$ hence
  $$\nu\left(([X_1\cdots X_N]C_N+\eps B_2^n)^{\circ}\right)\ls
  \nu(\eps^{-1}B_2^n)$$ for each $N\gr n$; the same holds for
  $Z_1,\ldots, Z_N$.  By dominated convergence, Lemmas
  \ref{lemma:Hausdorff}, \ref{lemma:converge} and Theorem
  \ref{thm:main_Lebesgue}, we have
  \begin{eqnarray*}
    \lefteqn{\EE\nu\left((\lim_{N\rightarrow
        \infty}[X_1\cdots X_N]C_N+\eps B_2^n)^{\circ}\right)}\\ & = &
    \EE\lim_{N\rightarrow \infty}\nu\left(([X_1\cdots
      X_N]C_N+\eps B_2^n)^{\circ}\right) \\ & = & \lim_{N\rightarrow \infty}
    \EE\nu\left(([X_1\cdots X_N]C_N+\eps
    B_2^n)^{\circ}\right) \\ & \ls &
    \lim_{N\rightarrow \infty} \EE\nu\left(([Z_1\cdots Z_N]C_N+\eps
    B_2^n)^{\circ}\right) \\ & = &\EE
    \lim_{N\rightarrow \infty}\nu\left(([Z_1\cdots Z_N]C_N+\eps
    B_2^n)^{\circ}\right) \\ & = &
    \EE\nu\left((\lim_{N\rightarrow
      \infty}[Z_1\cdots Z_N]C_N+\eps
    B_2^n)^{\circ}\right). 
  \end{eqnarray*}
  If $\EE\nu\left((\lim_{N\rightarrow \infty}[Z_1\cdots
    Z_N]C_N)^{\circ}\right)=\infty$,~(\ref{eqn:main_limit}) is
  trivial. Otherwise, since
  \begin{equation*}
    \lim_{N\rightarrow \infty}[Z_1\cdots Z_N]C_N +\eps B_2^n \supseteq
        \lim_{N\rightarrow \infty}[Z_1\cdots Z_N]C_N,
  \end{equation*}we have
  \begin{equation*} \nu\left((\lim_{N\rightarrow
      \infty}[Z_1\cdots Z_N]C_N+\eps B_2^n)^{\circ}\right)\ls
    \nu\left((\lim_{N\rightarrow \infty}[Z_1\cdots
      Z_N]C_N)^{\circ}\right)
    \end{equation*} 
for each $\eps >0$. Thus we can appeal to dominated convergence once
more and let $\eps\rightarrow 0$ to conclude the proof.
\end{proof}

Recall that given a measure $\mu\in \mathcal{P}_n$ and $p\gr 1$, the
$L_p$-centroid body $Z_p(\mu)$ of $\mu$ is the convex body with
support function
  \begin{equation*}
    h_{Z_p(\mu)}(y) = \left(\int_{\R^n} \abs{\langle x, y \rangle}^p
    d\mu(x)\right)^{1/p} \quad (y\in \R^n).
  \end{equation*}
  Such bodies were originally defined for compact star-shaped sets
  rather than measures, under an alternate normalization, in
  \cite{LZ_BS}.
    
  If the $X_i$'s are sampled according to $\mu$, then 
  \begin{equation}
    \label{eqn:Lpcentroid}
    Z_p(\mu) = \lim_{N\rightarrow \infty}N^{-1/p} [X_1\cdots
      X_N]B_q^N,
    \end{equation}
where $1/p+1/q=1$, and convergence occurs a.s. in $\delta^H$; the
latter follows from the law of large numbers (see
\cite{PaoPiv_probtake}). In particular, if $K$ is an origin-symmetric
convex body, then
\begin{equation}
 K = \lim_{N\rightarrow \infty}[X_1\cdots X_N]B_1^N =
 \lim_{N\rightarrow \infty}\conv{\{\pm X_1,\ldots,\pm X_N\}},
\end{equation}
where $X_1,X_2,\ldots$ are independent random vectors sampled in $K$
and convergence occurs a.s. in $\delta^H$.

\begin{corollary}
Let 
$\nu$ be a measure 
  on $\R^n$ with a spherically-symmetric, decreasing density.  
Let $\mu\in\mathcal{P}_n$,  $p\gr 1$, and  $Z_p(\mu)$ be the
$L_p$-centroid body of $\mu$.  Then, 
\begin{equation*}
  \nu(Z^{\circ}_p(\mu))\ls \nu(Z^{\circ}_p(\lambda_{D_n})).
\end{equation*}
\end{corollary}

When $\nu$ is Lebesgue measure on $\R^n$ and $\mu$ is the uniform
measure on a compact star-shaped set, the latter result is due to
Lutwak-Zhang \cite{LZ_BS}; a straightforward generalization from
star-shaped sets to measures $\mu$ appears in \cite{Paouris_super}.

\begin{proof} 
  By (\ref{eqn:Lpcentroid}), Theorem \ref{thm:apps} applies. 
\end{proof}

When $\nu$ is not Lebesgue measure, the result is very sensitive to
scaling, since we lose affine invariance. When we drop the volume
normalization, we can still prove the following result.

\begin{corollary}
  \label{newSan}
  Let $K$ be an origin-symmetric convex body in $\R^n$ and suppose
  that $\abs{K}=\abs{t_K B_2^n}$.  Then for any Lebesgue absolutely
  continuous measure $\nu$ with a spherically-symmetric,
  decreasing density, we have
  \begin{equation}
    \nu(K^{\circ})\ls \nu((t_KB_2^n)^{\circ}).
  \end{equation}
\end{corollary}

\begin{proof}
  Let $\overline{K} = K/\abs{K}^{1/n}$ be the volume one homothetic
  copy of $K$.  If $X_1,X_2,\ldots$ are independent random vectors
  sampled in $\overline{K}$, then $\overline{K} = \lim_{N\rightarrow
    \infty}[X_1\cdots X_N] B_1^N$ and hence
  \begin{equation*}
    K^{\circ} = \lim_{N\rightarrow \infty}\left([X_1\cdots
      X_N] B_1^N/\abs{K}^{1/n}\right)^{\circ},
  \end{equation*}where the convergence is a.s. in $\delta^H$.
  Similarly, if $Z_1,Z_2,\ldots$ are independent random vectors
  sampled in $D_n$, then we have a.s. convergence in $\delta^H$:
  \begin{equation*}
    D_n^{\circ} = \lim_{N\rightarrow
      \infty}([Z_1\cdots Z_N]B_1^N)^{\circ}.
  \end{equation*}
  Thus
  \begin{eqnarray*}
    \nu(K^{\circ}) & = &\EE
      \nu\left(\lim_{N\rightarrow \infty}([X_1\cdots
        X_N] B_1^N/\abs{K}^{1/n})^{\circ}\right)\\
       & \ls & \EE
      \nu\left(\lim_{N\rightarrow \infty}([X_1\cdots
        X_N] B_1^N/\abs{K}^{1/n})^{\circ}\right)\\
      & = & \nu(\abs{K}^{-1/n}D_n^{\circ})\\
      & = & \nu((t_K B_2^n)^{\circ}).
  \end{eqnarray*}
\end{proof}

\section*{Acknowledgements}

Dario Cordero-Erausquin and Matthieu Fradelizi are supported in part
by the Agence Nationale de la Recherche, project GeMeCoD (ANR 2011
BS01 007 01); Grigoris Paouris by the A. Sloan foundation, BSF grant
2010288 and the US NSF grant CAREER-1151711; Peter Pivovarov by the
University of Missouri Research Board.  Lastly, we are grateful for
the hospitality of Universit\'e Paris VI, which facilitated part of
our collaboration.


\bigskip

\noindent
Dario Cordero-Erausquin: Institut de Math\'ematiques de Jussieu-PRG, Universit\'e Pierre et Marie Curie (Paris 6),
75252 Paris Cedex 05, France \\
cordero@math.jussieu.fr

\medskip

\noindent 
Matthieu Fradelizi:
Laboratoire d'Analyse et de Math\'ematiques
Appliqu\'ees, Universit\'e Paris-Est Marne la Vall\'ee, 77454 Marne la Vall\'ee Cedex 2, France \\
matthieu.fradelizi@univ-mlv.fr 

\medskip

\noindent 
Grigoris Paouris:
 Department of Mathematics (Mailstop 3368), Texas A\&M University,
College Station, TX 77843-3368, USA\\
grigoris@math.tamu.edu 

\medskip

\noindent 
Peter Pivovarov: Mathematics Department,
University of Missouri,
Columbia, MO 65211, USA\\
pivovarovp@missouri.edu

\end{document}